\documentclass{article}

\usepackage{journal2e_v1}
\usepackage[numbers]{natbib}

\usepackage[utf8]{inputenc} 
\usepackage[T1]{fontenc}    
\usepackage{url}            
\usepackage{booktabs}       
\usepackage{amsfonts}       
\usepackage{nicefrac}       
\usepackage{microtype}      

\usepackage{epsf}
\usepackage{epsfig}
\usepackage{amsmath}
\usepackage{subfigure}
\usepackage{amssymb}
\usepackage{multirow}
\usepackage{multicol}
\usepackage{algorithmic}

\def\A{{\bf A}}

\def\B{{\bf B}}

\def\D{{\bf D}}

\def\G{{\bf G}}

\def\I{{\bf I}}

\def\M{{\bf M}}

\def\S{{\bf S}}
\def\s{{\bf s}}

\def\U{{\bf U}}

\def\V{{\bf V}}
\def\v{{\bf v}}

\def\X{{\bf X}}
\def\x{{\bf x}}

\def\Z{{\bf Z}}

\def\A{{\bf A}}

\def\B{{\bf B}}

\def\D{{\bf D}}

\def\G{{\bf G}}

\def\I{{\bf I}}

\def\M{{\bf M}}

\def\N{{\bf N}}

\def\S{{\bf S}}
\def\s{{\bf s}}

\def\U{{\bf U}}

\def\V{{\bf V}}
\def\v{{\bf v}}

\def\X{{\bf X}}
\def\x{{\bf x}}

\def\Z{{\bf Z}}

\def\0{{\bf 0}}
\def\1{{\bf 1}}

\def\OM{{\mathcal O}}

\def\RB{{\mathbb R}}
\def\RBmn{{\RB^{m\times n}}}

\def\Pii{\mbox{\boldmath$\Pi$\unboldmath}}

\def\Si{\mbox{\boldmath$\Sigma$\unboldmath}}

\def\Ome{\mbox{\boldmath$\Omega$\unboldmath}}

\def\argmin{\mathop{\rm argmin}}

\def\nnz{\mathrm{nnz}}

\def\tr{\mathrm{tr}}
\def\rk{\mathrm{rank}}
\def\diag{\mathrm{diag}}

\def\RB{{\mathbb R}}

\title{Tighter bound of Sketched Generalized Matrix Approximation}

\author{\name  Haishan Ye, Qiaoming Ye   \\
	\addr {yhs12354123@gmail.com; yqmlilian@gmail.com } \\
	Department of Computer Science and Engineering \\
	Shanghai Jiao Tong University \\
	800 Dong Chuan Road, Shanghai, China 200240	
	\AND
	\name Zhihua Zhang   \\
	\addr zhzhang@gmail.com \\
	School of Mathematical Science \\
	Peiking  University \\
	Beijing, China 100871
}

\begin{document}

\maketitle

\begin{abstract}
	Generalized matrix approximation plays a fundamental  role in many 
	machine learning problems, such as CUR decomposition, kernel 
	approximation, and matrix low rank approximation. Especially with 
	Today's applications involved in larger and larger dataset, more and 
	more efficient generalized matrix approximation algorithems become
	a crucially important research issue.
	In this paper, we find new sketching techniques to reduce the size of 
	the original data matrix to develop new matrix approximation algorithms.
	Our results derive a much tighter bound for the approximation than 
	previous works: we obtain a $(1+\epsilon)$ approximation ratio with small sketched dimensions  which
	implies a more efficient  generalized matrix approximation.   
\end{abstract}

\section{Introduction}

Matrix manipulations are the basis of modern data analysis. As the datasets becomes larger and larger, it is much more difficult to perform exact matrix multiplication, inversion, and  decomposition. Consequently, matrix approximation techniques have been extensively studied, including approximate matrix multiplication \cite{cohen1999approximating,drineas2006fast,kyrillidis2014approximate} and low-rank matrix approximation \cite{bourgain2013toward,boutsidis2014near,clarkson2013low,martinsson2011randomized}. 

In this paper we are concerned with the generalized matrix approximation problem \cite{stewart1999four,friedland2007generalized,sou2012generalized}: 
\[ \min_{\X}\|\A-\M\X\N\|_F,\] 
where $\A\in\RBmn$, $\M\in\RB^{m\times c}$ and $\N\in\RB^{r\times n}$. It is well known that the solution is $\X^* = \M^{\dagger}\A\N^{\dagger}$. It costs $\OM(\nnz(\A)\cdot\min(c,r)+mc^2+nr^2)$ time to solve the generalized matrix approximation exactly to get $\X^{*}$. Generalized matrix approximation takes an important role in solving some machine learning problems such as the CUR decomposition \cite{wang2013improving,gittens2013revisiting,wang2015towards}, modified Nystr{\"o}m method \cite{wang2013improving}, and distributed PCA \cite{liang2014improved,boutsidis2015optimal}. It is also a key research topic in numerical linear algebra \cite{stewart1999four,friedland2007generalized}.

When $\N$ is the identity matrix, the generalized matrix approximation degenerates to the ordinary least squares regression. To solve the least squares regression $\tilde{\X} = \argmin_\X\|\M\X - \A\|_F$ more efficiently when $c\ll m$, recent studies suggest multiplying a sketching matrix $\S\in\RB^{s_c\times m}$, where $s_c = \OM(c/\epsilon)$, to get a sketched least squares regression $\hat{\X} = \argmin_\X\|\S(\M\X - \A)\|_F$. The studies~\cite{clarkson2013low,drineas2011faster,sarlos2006improved} also prove that the reduced least squares regression obtains a $(1+\epsilon)$ relative error bound, which is, $\|\M\hat{\X}-\A\|_F\leq(1+\epsilon)\|\M\tilde{\X}-\A\|_F$. 

Since the ordinary least squares regression is a special case of the generalized matrix approximation, it is reasonable to use the sketching technique to accelerate solving the generalized matrix approximation.   Recently, \citet{wang2015towards} proposed to use leverage-score sketching matrices $\S_\M\in \RB^{s_c\times m}$ and $\S_\N\in\RB^{s_r\times n}$ to efficiently solve the generalized matrix approximation via
\[
\hat{\X} = \argmin_\X \|\S_\M(\A-\M\X\N)\S_\N^T\|_F.
\]
However, it needs $s_c = \OM(\sqrt{\min(m,n)}c/\sqrt{\epsilon})$ and $s_r = \OM(\sqrt{\min(m,n)}r/\sqrt{\epsilon})$ to get a $(1+\epsilon)$ relative error bound in \cite{wang2015towards}. The result of \cite{wang2015towards} can be easily extended to other sketching matrices with the same reduced dimensions, i.e., the same $s_c$ and $s_r$. \citet{boutsidis2015optimal} proposed to solve the generalized rank-constrained matrix approximation problem $\min_{\text{rank}(\X)\leq k}\|\A-\M\X\N\|_F$ by multiplying sketching matrices $\S_\M$ and $\S_\N$ with the affine embedding property. It needs $s_c = \OM(c/\epsilon^2)$ and $s_r = \OM(r/\epsilon^2)$ to achieve a $(1+\epsilon)$ relative error bound. 

In this paper, we prove that  $s_c = \OM(c/\epsilon)$ and $s_r = \OM(r/\epsilon)$ are enough to achieve a $(1+\epsilon)$ relative error bound for the generalized matrix approximation problem. We compare our result with previous work in Table~\ref{tb:comp} in detail. As we can see, our bound is much tighter than previous work. Besides, a tigher bound leads to faster sketched generalized matrix approximation. Using leverage-score sketching, the running time in Theorem~\ref{thm:gmp_lev_score} is much less than that of \citet{boutsidis2015optimal} since, $c$ and $r$ is much less than $n$ and $m$, especially when $\A$ is dense. Comparing with \cite{wang2015towards}, though the first part of compuation cost both is $\OM(\nnz(\M))$, the running time is $m$ times less than \cite{wang2015towards} as to the second part of computation cost. Hence, our result is better, especially when $\nnz(\M)$ is small, because the leading cost will be the second part of computation cost.   
\begin{table}[]
	\centering
	\caption{Comparisons of skecthed dimensions (Assume $m\leq n$ and $\nnz(\M)\leq \nnz(\N)$)}
	\label{tb:comp}
	\begin{tabular}{|c|c|c|l|}
		\hline
		Reference                    & $s_c$                   & $s_r$                                           & Computation Cost         \\ \hline
		\citet{boutsidis2015optimal} & $\OM(c/\epsilon^2)$     & $\OM(r/\epsilon^2)$                             & $\OM(\nnz(\A) + \nnz(\M)+poly(c,r,\epsilon^{-1}))$             \\ \hline
		\citet{wang2015towards}      & $\OM(\sqrt{m}c/\sqrt{\epsilon})$    & $\OM(\sqrt{m}c/\sqrt{\epsilon})$    &  $\OM(\nnz(\M)+m\cdot poly(c,r,\epsilon^{-1})$            \\ \hline
		Theorem~\ref{thm:gmp_sps_gauss}   & $\OM(c/\epsilon)$    & $\OM(r/\epsilon)$                             & $\OM(\nnz(\A)+\nnz(\M)+poly(c,r,\epsilon^{-1}))$            \\ \hline
		Theorem~\ref{thm:gmp_lev_score}   & $\OM(c/\epsilon+c\log c)$    &$\OM(r/\epsilon+r\log r)$    &        $\OM(\nnz(\M)+poly(c,r,\epsilon^{-1}))$      \\ \hline
	\end{tabular}
\end{table}
\section{Notation and Preliminaries}
\label{sec:notation}

In this section, we introduce some notation and  preliminaries
that will be used in this paper. 

\subsection{Notation}

Let $\I_m$ be the $m{\times}m$ identity matrix. 
Given a matrix $\A=[a_{ij}] \in \RB^{m \times n}$ of rank $ \rho $,  its SVD is given as
$\A=\U\Si\V^{T}=\U_{k} \Si_{k} \V_{k}^{T}+\U_{\rho\setminus k} \Si_{\rho{\setminus} k}\V_{\rho{\setminus}k}^{T}$,
where $\U_{k}$ and $\U_{\rho{\setminus}k}$ contain the left singular vectors of $\A$,  $\V_{k}$ and $\V_{\rho{\setminus}k}$ contain the right
singular vectors of $\A$, and $\Si=\diag(\sigma_1, \ldots, \sigma_{\rho})$ with $\sigma_1\geq \sigma_2 \geq \cdots \geq \sigma_{\rho}>0$ are
the nonzero singular values of $\A$. Accordingly,   $\|\A\|_{F} \triangleq (\sum_{i,j}a_{ij}^{2})^{1/2}=(\sum_{i}\sigma_{i}^{2})^{1/2}$
is the Frobenius norm of $\A$ and
$\|\A\|_{2}\triangleq \sigma_{1}$ is the spectral norm.

Additionally, $\A^{\dagger} \triangleq \V\Si^{-1}\U^{T} \in \RB^{n \times m}$ is
the  Moore-Penrose pseudoinverse of $\A$, which is  unique. It is easy to verify that ${\rk}(\A^{\dagger}) ={\rk}(\A)=\rho$. 
Moreover, for all $i=1, \dots, \rho$,
$\sigma_{i}(\A^{\dagger})=1/\sigma_{\rho-i+1}(\A)$. If  $\A$ is of full row rank, then $\A\A^{\dagger}=\I_{m}$.
Also, if $\A$ is of full column rank, then $\A^{\dagger}\A=\I_{n}$. 
When $m=n$, $\text{tr}(\A) = \sum_{i=1}^{n}a_{ii}$ is the trace of $\A$.

It is well known that $\A_{k}\triangleq \U_{k}\Si_{k}\V_{k}^{T}$ is the minimizer of both $\min \|\A-\X\|_{F}$ and $\min \|\A-\X\|_{2}$ over all matrices $\X \in \RB^{m \times n}$ of rank at most $k\leq\rho$. Thus, $\A_k$ is called the best rank-$k$ approximation of $\A$.
Let $\nnz(\A)$ denote the number of nonzero entries of $\A$.  


\subsection{Subspace Embedding}

Oblivious subspace embedding is an important sketching tool in randomized numerical linear algebra. 
By oblivious subspace embedding, a matrix can be projected to a much lower dimensional subspace, 
which leads to much faster matrix operations. 
\begin{definition}[\cite{woodruff2014sketching}]
	Given $\varepsilon>0$ and $\delta>0$, let $\Pi$ be a distribution on $s \times m$ real matrices, where $s$ relies on $m$, $d$, $\varepsilon$ and $\delta$. Suppose that
	with probability at least $1-\delta$, for any fixed $m \times d$ matrix $\A$, a matrix $\S$ drawn from distribution $\Pi$
	is a $(1+\epsilon)$ $\ell_{2}$-subspace embedding for $\A$, that is, for all $\x \in \RB^{d}$,
	$\|\S\A\x\|_{2}^{2}=(1\pm\epsilon)\|\A\x\|_{2}^{2}$ with probability
	$1-\delta$. Then we call $\Pi$ an $(\epsilon,\delta)$-oblivious $\ell_{2}$-subspace
	embedding. 
\end{definition}
For the sake of conciseness, the $(\epsilon,\delta)$-oblivious $\ell_{2}$-subspace embedding is referred as an $\epsilon$-subspace embedding. Now we list some important subspace embedding matrices and their properties which will be used in this paper.

\begin{definition}[leverage-score sketching matrix] \label{def:lev_score}
	Let $\V\in\RB^{n\times k}$  be column orthonormal basis for $\A\in\RB^{n\times k}$  with $n>k$,  and $\v_{i,*}$ denote the $i$-th row of $\V$. Let $\ell_{i} = \|\v_{i,*}\|_{F}^{2}/k$ and $r$ be an integer with $1\leq r\leq n$. Then the $\ell_{i}$'s are leverage scores for $\A$. Construct a sampling matrix  $\Ome\in\RB^{n\times r}$ and a rescaling matrix $\D\in\RB^{r \times r}$ as follows. For every $j = 1,\dots,r$, independently and with replacement, pick an index $i$ from the set $\{1,2\dots,n\}$ with probability  $\ell_{i}$ and set $\Ome_{ij} = 1$ and $\D_{jj}=1/\sqrt{\ell_{i}r}$. The leverage-score sketching matrix $\S$ for $\A$ is then defined as $\S = \Ome\D$. 
\end{definition}

\begin{theorem}[\cite{woodruff2014sketching}]\label{thm:subspace_embedding}
	Given $\A\in\RB^{m\times d}$ of full column rank, assume $\S\in \RB^{s\times m}$ is an $\epsilon$-subspace embedding matrix for $\A$. If $\S$ is a sparse subspace embedding matrix, then $s = \OM(d^2\epsilon^{-2})$ is sufficient. If $\S$ is an  $s \times m$ matrix of i.i.d.\ normal random variables with variance $1/s$, then $s=\Theta(d\epsilon^{-2})$ is sufficient. For a leverage-score sketching matrix $\S$, $s = \OM(d\log d\epsilon^{-2})$ is needed.
\end{theorem}

\begin{theorem}[\cite{clarkson2013low,boutsidis2015optimal}]\label{thm:multi}
	For $\A\in \RBmn$ and  $\B\in \RB^{m \times k}$, there is an $s = \Theta(\epsilon^{-2})$, so that for an $s \times m$ sparse embedding matrix $\S$ or an $s \times m$ matrix $\S$ of i.i.d.\ normal random variables with variance $1/s$, or an $s\times m$ leverage-score sketching matrix for $\A$ under the condition that $\A$ has orthonormal columns, then it holds that 
	\begin{equation}
		\|\A^{T}\S^{T}\S\B-\A^{T}\B\|_{F}^{2}<\epsilon^{2}\|\A\|_{F}^{2}\|\B\|_{F}^{2} \label{eq:multi}
	\end{equation}
	with probability at least $1-\delta$	for any fixed $\delta>0$.
\end{theorem}

Other types of sketching matrices like Subsampled Randomized Hadamard Transformation and detailed properties of sketching matrices and subspace embedding matrices can be
found in the survey~\cite{woodruff2014sketching}.
\section{Main Result}
\label{sec:main_result}

We first give the conditions that subspace embedding matrices should satisfy for a sketched generalized matrix approximation to achieve a $(1+\epsilon)$ error bound. The detailed conditions are depicted in Theorem~\ref{thm:gmp_main}.

\begin{theorem}\label{thm:gmp_main}
	Given that $\A\in \RB^{m\times n}$, $\M\in\RB^{m\times c}$ and $\N\in \RB^{r\times n}$, assume that $\S_\M\in \RB^{s_c\times m}$ is a subspace embedding matrix for $\M$ with error parameter $\epsilon_0 = 1/2$,  and $\S_\M$ also makes  Eqn.~\eqref{eq:multi} hold with error parameter $\sqrt{\epsilon/c}$. Similarly, $\S_\N\in\RB^{\s_r\times n}$ is a subspace embedding matrix for $\N^T$ with  error parameter $\epsilon_0 = 1/2$ and $\S_\N$  also makes Eqn.~\eqref{eq:multi} hold with error parameter $\sqrt{\epsilon/r}$.  Let
	\begin{equation}
		\X^{*} =  \M^{\dagger}\A\N^{\dagger} = \argmin_{\X}\|\A-\M\X\N\|_{F} \label{eq:U*}
	\end{equation}
	and 
	\begin{equation}
		\hat{\X} = (\S_{\M}\M)^{\dagger}\S_{\M}\A\S_{\N}^T(\N\S_{\N}^T)^{\dagger} \label{eq:U_hat}.
	\end{equation}
	Then we have
	\[
	\|\A-\M\hat{\X}\N\|_{F}\leq (1+\epsilon)\|\A-\M\X^{*}\N\|_{F}.
	\]
\end{theorem}

The conditions required in Theorem~\ref{thm:gmp_main} are all satisfied by oblivious embedding matrices, including sparse embedding matrices, gaussian matrices, subsampled randomized Hadamard matrices~\cite{woodruff2014sketching}, as well as their combinations. Now we present  Theorem~\ref{thm:gmp_sps_gauss}, which shows that  an $\nnz(\A)$ generalized matrix approximation solver can be achieved and the corresponding sketched dimensions are $s_c = \OM(c/\epsilon)$ and $s_r = \OM(r/\epsilon)$, respectively. 

\begin{theorem}\label{thm:gmp_sps_gauss}
	We are given $\A\in \RB^{m\times n}$, $\M\in\RB^{m\times c}$ and $\N\in \RB^{r\times n}$. Assume that $\Pii_\M$ is a $t\times m$ sparse embedding matrix and $\G_\M$ is an $s_c\times t$ matrix of i.i.d.\ normal random variables with variance $1/s_c$, where $t = \OM(c/\epsilon+c^2)$ and $\s_c = \OM(c/\epsilon)$. Construct $\S_\M = \G_\M\Pii_\M$. Similarly, $\Pii_\N$ is a $t^{\prime}\times n$ sparse embedding matrix and $\G_\N$ is a $s_r\times t^{\prime}$ matrix of i.i.d.\ normal random variables with variance $1/s_r$, where $t^\prime = \OM(r/\epsilon+r^2)$ and $s_r = \OM(r/\epsilon)$. Then we construct $\S_{\N} = \G_{\N}\Pii_{\N}$. If $\X^{*}$ and $\hat{\X}$ are defined as \eqref{eq:U*} and \eqref{eq:U_hat} respectively, then we have
	\[
	\|\A-\M\hat{\X}\N\|_{F}\leq (1+\epsilon)\|\A-\M\X^{*}\N\|_{F},
	\]
	with high probability and $\hat{\X}$ can be constructed with the computational complexity of 
	\begin{align}
		\OM(&\nnz(\A)+mc+nr+ c^2r/\epsilon^3 +cr^2/\epsilon^3 + c^2r^2/\epsilon^2 \notag+ c^3r/\epsilon^2\\&+cr^3/\epsilon^2 +\min(c^2r^3/\epsilon, c^3r^2/\epsilon))+c^4/\epsilon+r^4/\epsilon). \label{eq:time_analysis}
	\end{align}
\end{theorem}

\begin{proof}
	Let $\Pii_\M$ and $\G_\M$ be $\epsilon_0$-subspace embedding matrices for $c$-subspace with $\epsilon_0=1/4$, which needs $t = \OM(c^2)$ and $s_c = \OM(c)$. By Lemma~\ref{lem:embed_add}, we have $\G_\M\Pii_\M$ is a $1/2$-subspace embedding matrix for $\M$. By Lemma~\ref{lem:mult_add} and Theorem~\ref{thm:multi}, $t = \OM(c/\epsilon)$  and $s_c = \OM(c/\epsilon)$, and $\S_\M = \G_\M\Pii_\M$ makes Eqn.~\eqref{eq:multi} hold with error parameter $\sqrt{\epsilon/c}$. The similar result holds for $\Pii_\N$, $\G_\N$ and $\S_\N$. Thus, $\S_\M$ and $\S_\N$ satisfy all conditions in Theorem~\ref{thm:gmp_main}. By Theorem~\ref{thm:gmp_main}, we obtain the result.
	
	For computation complexity, it needs $\OM(\nnz(\A)+ c^2r/\epsilon^3 +cr^2/\epsilon^3 + c^2r^2/\epsilon^2+ c^3r/\epsilon^2+cr^3/\epsilon^2 +\min(c^2r^3/\epsilon, c^3r^2/\epsilon))$ to compute $\S_\M\A\S_\N^T$. Computing $(\S_\M\M)^{\dagger}$ and $(\N\S_\N^T)^{\dagger}$ requires $\OM(mc + c^3/\epsilon^2 + c^4/\epsilon)$ and $\OM(nr+r^3/\epsilon^2+r^4/\epsilon)$ respectively. Matrix multiplications of  $(\S_{\M}\M)^{\dagger}\S_{\M}\A\S_{\N}^T(\N\S_{\N}^T)^{\dagger}$ need $\OM(\min(c^2r/\epsilon^2+cr^2/\epsilon, cr^2/\epsilon^2+c^2r/\epsilon))$. Hence, the total cost of constructing $\hat{\X}$ is $\OM(\nnz(\A)+mc+nr+ c^2r/\epsilon^3 +cr^2/\epsilon^3 + c^2r^2/\epsilon^2+ c^3r/\epsilon^2+cr^3/\epsilon^2 +\min(c^2r^3/\epsilon, c^3r^2/\epsilon))+c^4/\epsilon+r^4/\epsilon)$. 
\end{proof}
The lemmas mentioned in the proof of Theorem~\ref{thm:gmp_sps_gauss} are given in Appendix~\ref{app:lemma}.

Leverage-score sketching matrices are significant in randomized numerical linear algebra. Using leverage-score sketching matrices, we can achieve faster sketched generalized matrix approximation than $\OM(\nnz(\A))$. It just needs $\OM(\nnz(\M)+\nnz(\N))$ arithmetic operations comparing with $\OM(\nnz(\A))$ in Theorem~\ref{thm:gmp_sps_gauss}.

\begin{theorem} \label{thm:gmp_lev_score}
We are given $\A\in \RB^{m\times n}$, $\M\in\RB^{m\times c}$ and $\N\in \RB^{r\times n}$. Let $\S_\M\in \RB^{s_c\times m}$ be the leverage-score sketching matrix for $\M$  with $s_c = \OM(c/\epsilon+c\log c)$. Similarly, $\S_\N\in\RB^{s_r\times n}$ is the leverage-score sketching matrix for $\N^T$ with $s_r = \OM(r/\epsilon+r\log r)$. If $\X^{*}$ and $\hat{\X}$ are defined as \eqref{eq:U*} and \eqref{eq:U_hat} respectively, then we have
	\[
	\|\A-\M\hat{\X}\N\|_{F}\leq (1+\epsilon)\|\A-\M\X^{*}\N\|_{F}
	\]
	with high probability. And $\hat{\X}$ can be constructed with  the computational complexity of
	\begin{align*}
	\OM(&\nnz(\M)+\nnz(\N) + (c^2r+cr^2)/\epsilon^2+(c^2r+cr^2)\log(cr)/\epsilon \\&+ (c^3+r^3)/\epsilon+(c^2r+cr^2)\log c \log r+ c^3\log c+r^3\log r)
	\end{align*}	
\end{theorem} 
\begin{proof}
	By Theorem~\ref{thm:subspace_embedding} and~\ref{thm:multi}, it is easy to check that $\S_\M$ and $\S_\N$ satisfy the conditions in Theorem~\ref{thm:gmp_main}. Hence, the result holds by Theorem~\ref{thm:gmp_main}. 
	
	As for compuational cost, it costs $\OM(\nnz(\M))$ and $\OM(\nnz(\N))$ time to compute leverage scores of $\M$ and $\N^T$ respectively~\cite{li2013iterative}. And it takes $\OM(cr/\epsilon^2+cr\log(cr)/\epsilon + \log c\log r)$ to compute $\S_\M\A\S_\N^T$. And it requires $\OM(c^3/\epsilon+r^3/\epsilon + c^3\log c+r^3\log r)$ time to compute $(\S_\M\M)^{\dagger}$ and $(\S_\N\N^T)^{\dagger}$. It costs $\OM((c^2r+cr^2)/\epsilon^2 + (c^2r+cr^2)\log(cr)/\epsilon+(c^2r+cr^2)\log c \log r)$ arithmetic operations to achieve the multiplication of $(\S_{\M}\M)^{\dagger}\S_{\M}\A\S_{\N}^T(\N\S_{\N}^T)^{\dagger}$. Hence, the total cost of constructing $\hat{X}$ is 
	\begin{align*}
	\OM(&\nnz(\M)+\nnz(\N) + (c^2r+cr^2)/\epsilon^2+(c^2r+cr^2)\log(cr)/\epsilon \\&+ (c^3+r^3)/\epsilon+(c^2r+cr^2)\log c \log r+ c^3\log c+r^3\log r)
	\end{align*}
\end{proof}

Now we consider the symmetric case where $\A$ is symmetric and $\N = \M^T$,  $\X^*$ constructed as Eqn.~\eqref{eq:U*} is a symmetric matrix. Note that $\hat{\X}$ is asymmetric in most cases since $\S_\M$ and $\S_\N$ are chosen independently. However, we can construct $\tilde{\X} = (\hat{\X}+\hat{\X}^T)/2$ which is symmetric and can still keep relative error bound. 
\begin{corollary}
	Let $\A\in\RB^{m\times m}$ be a symmetric matrix. $\M$ is an $m\times c$ matrix. $\S_1$ and $\S_2$ are subspace embedding matrices for $\M$ with error parameter $\epsilon_0 = 1/2$,  and they also satisfy Eqn.~\eqref{eq:multi}  with error parameter $\sqrt{\epsilon/c}$.  Let
	\begin{equation*}
		\X^{*} =  \M^{\dagger}\A(\M^T)^{\dagger} = \argmin_{\X}\|\A-\M\X\M^T\|_{F}
	\end{equation*}
	and 
	\begin{equation*}
		\hat{\X} = (\S_{1}\M)^{\dagger}\S_{1}\A\S_{2}^T(\M^T\S_{2}^T)^{\dagger}.
	\end{equation*}
	Construct a symmetric matrix $\tilde{\X}$ as
	\[
	\tilde{\X} =(\hat{\X}+\hat{\X}^T)/2.
	\]
	Then we have
	\[
	\|\A-\M\tilde{\X}\M^T\|_{F}\leq (1+\epsilon)\|\A-\M\X^{*}\M^T\|_{F}.
	\]
\end{corollary}

\section{Proof of Main Theorem}\label{sec:main_proof}
In this section, we give the detailed proof of our main theorem, i.e., Theorem~\ref{thm:gmp_main}.
\begin{proof} {\bf{of Theorem~\ref{thm:gmp_main}}}
	We define 
	\[
	\A^\vdash \triangleq \A-\M\X^*\N.
	\]
	We let the condensed SVDs of $\M$ and $\N$ be respectively
	\[
	\M = \X_\M\Si_{\M}\V_\M^T \qquad \text{and } \qquad \N = \X_\N\Si_\N\V_\N^T.
	\]
	
	We have 
	\begin{align}
	\M\hat{\X}\N = & \M(\S_{\M}\M)^{\dagger}\S_{\M}\A\S_{\N}^T(\N\S_{\N}^T)^{\dagger}\N \notag\\
	=& \U_\M\Si_{\M}\V_\M^T(\S_\M\U_\M\Si_{\M}\V_\M^T)^\dagger\S_{\M}\A\S_{\N}^T(\U_\N\Si_\N\V_\N^T\S_\N^T)^\dagger\U_\N\Si_\N\V_\N^T \notag\\
	=&\U_\M\Si_{\M}\V_\M^T(\Si_{\M}\V_\M^T)^\dagger(\S_\M\U_\M)^\dagger\S_{\M}\A\S_{\N}^T(\V_\N^T\S_\N^T)^\dagger(\U_\N\Si_\N)^{\dagger}\U_\N\Si_\N\V_\N^T \label{eq:pinv}\\
	=& \U_\M(\S_\M\U_\M)^\dagger\S_{\M}\A\S_{\N}^T(\V_\N^T\S_\N^T)^\dagger\V_\N^{T}, \notag
	\end{align}
	where \eqref{eq:pinv} is because $\S_\M\U_\M$ is of full column rank and $\Si_\M\V_\M^T$ is of full row rank.  
	
	We define $\Z^*$ by
	\[
	\M\X^*\N = \U_\M(\Si_{\M}\V_\M^T\X^*\U_\N\Si_\N)\V_\N^T \equiv \U_\M\Z^*\V_\N^T.
	\]
	Similarly, we define $\hat{\Z}$ by
	\[
	\M\hat{\X}\N \equiv \U_\M\hat{\Z}\V_\N^T.
	\]
	Then, we have that
	\begin{align*}
	&(\S_{\M}\U_\M)^T(\S_{\M} \U_\M)\hat{\Z}(\V_\N^{T}\S_\N^T)(\V_\N^{T}\S_\N^T)^T \\
	&=(\S_{\M}\U_\M)^T\S_{\M}\A\S_{\N}^T(\V_\N^{T}\S_\N^T)^T\\
	&=(\S_{\M}\U_\M)^T\S_{\M}(\A^\vdash+\U_\M\Z^*\V_\N^T)\S_{\N}^T(\V_\N^{T}\S_\N^T)^T.
	\end{align*}
	Hence, we have
	\begin{align*}
	(\S_{\M}\U_\M)^T\S_{\M}\A^\vdash\S_{\N}^T(\V_\N^{T}\S_\N^T)^T = & (\S_{\M}\U_\M)^T\S_{\M}\U_\M(\hat{\Z}-\Z^*)\V_\N^T\S_{\N}^T(\V_\N^{T}\S_\N^T)^T \\
	=&  (\S_{\M}\U_\M)^T\S_{\M}\U_\M\Z\V_\N^T\S_{\N}^T(\V_\N^{T}\S_\N^T)^T,
	\end{align*}
	where we define $\Z = \hat{\Z}-\Z^*$.
	
	Since $\S_\M\U_\M$ is of full column rank and $s_c\geq c\geq \rho_c$, where $\rho_c$ is the rank of $\M$, we have that $(\S\U_\M)^T\S_{\M}\U_\M$ is nonsingular, Similarly, $\V_\N^T\S_{\N}^T(\V_\N^{T}\S_\N^T)^T$ is nonsingular. We obtain
	\begin{align*}
	\Z = [(\S_\M\M)^T\S_{\M}\U_\M]^{-1}[(\S_\M\U_\M)^T\S_{\M}\A^\vdash\S_{\N}^T(\V_\N^{T}\S_\N^T)^T] [\V_\N^T\S_{\N}^T(\V_\N^{T}\S_\N^T)^T]^{-1}
	\end{align*}
	and thus
	\begin{align*}
	\|\Z\|_F \leq& \|[(\S_\M\U_\M)^T\S_{\M}\U_\M]^{-1}\|_2\|[\V_\N^T\S_{\N}^T(\V_\N^{T}\S_\N^T)^T]^{-1}\|_2\|[(\S_\M\U_\M)^T\S_{\M}\A^\vdash\S_{\N}^T(\V_\N^{T}\S_\N^T)^T]\|_F\\
	=& \sigma_{\min}^{-2}(\S_\M\U_\M)\sigma_{\min}^{-2}(\S_\N\V_\N)\cdot\|[(\S_\M\U_\M)^T\S_{\M}\A^\vdash\S_{\N}^T(\V_\N^{T}\S_\N^T)^T]\|_F.
	\end{align*}
	
	We can expand $\|\A-\M\hat{\U}\N\|_F^{2}$ as follows:
	\begin{align}
	&\|\A-\M\hat{\X}\N\|_F^{2}  \notag\\
	=& \|\A-\M\X^*\N+\M\X^*\N-\M\hat{\X}\N\|_F^2 \notag\\
	=& \|\A-\M\X^*\N\|_F^{2}+\|\M\X^*\N-\M\hat{\X}\N\|_F^{2} + 2\tr[(\A-\M\X^*\N)^T(\M\X^*\N-\M\hat{\X}\N)]  \label{eq:orth}\\
	=& \|\A-\M\X^*\N\|_F^{2} + \|\U_\M\Z\V_\N^{T}\|_F^{2}, \notag
	\end{align}
	where \eqref{eq:orth} is because
	\begin{align*}
	&\tr[(\A-\M\X^*\N)^T(\M\X^*\N-\M\hat{\X}\N)] \\
	=& \tr[((\I_m-\M\M^{\dagger})\A+(\M\M^{\dagger})\A(\I_n-\N^\dagger\N))^T\M(\X^*-\hat{\X})\N]\\
	=& \tr[\A^{T}(\I_m-\M\M^{\dagger})\M(\X^*-\hat{\X})\N] + \tr[\N(\I_n-\N^{\dagger}\N)\A^T\M\M^{\dagger}\M(\X^*-\hat{\X})] \\
	=& 0.
	\end{align*}
	Now we need to bound $\|\Z\|_F$. First, we express $\A^{\vdash}$ as follow
	\begin{align*}
	\A^{\vdash} = \A - \M\X^*\N = \U_\M^{\perp}(\U_\M^{\perp})^T\A + \U_\M\U_\M^T\A\V_\N^{\perp}(\V_{\N}^{\perp})^T.
	\end{align*}
	Then, we have
	\begin{align*}
	&\|[(\S_\M\U_\M)^T\S_{\M}\A^\vdash\S_{\N}^T(\V_\N^{T}\S_\N^T)^T]\|_F \\
	\leq& \|(\S_\M\U_\M)^T\S_{\M}\U_\M^{\perp}(\U_\M^{\perp})^T\A\S_{\N}^T(\V_\N^{T}\S_\N^T)^T\|_F + \|(\S_\M\U_\M)^T\S_{\M}\U_\M\U_\M^T\A\V_\N^{\perp}(\V_{\N}^{\perp})^T\S_{\N}^T(\V_\N^{T}\S_\N^T)^T\|_F.
	\end{align*}
	
	Since $\U_\M^T\U_\M^{\perp}(\U_\M^{\perp})^T\A\S_{\N}^T(\V_\N^{T}\S_\N^T)^T = \0$, and	by Theorem~\ref{thm:multi}, we have	
	\begin{align}
	&\|(\S_\M\U_\M)^T\S_{\M}\U_\M^{\perp}(\U_\M^{\perp})^T\A\S_{\N}^T(\V_\N^{T}\S_\N^T)^T\|_F \notag\\
	= & \|\U_\M^T\S_\M^T\S_{\M}\U_\M^{\perp}(\U_\M^{\perp})^T\A\S_{\N}^T(\V_\N^{T}\S_\N^T)^T - \U_\M^T\U_\M^{\perp}(\U_\M^{\perp})^T\A\S_{\N}^T(\V_\N^{T}\S_\N^T)^T\|_F\notag\\
	\leq&\frac{\sqrt{\epsilon}}{\sqrt{c}}\|\U_\M^T\|_F\|\U_\M^{\perp}(\U_\M^{\perp})^T\A\S_{\N}^T(\V_\N^{T}\S_\N^T)^T\|_F\label{eq:app_mult}\\
	=& \sqrt{\epsilon}\|\U_\M^{\perp}(\U_\M^{\perp})^T\A\S_{\N}^T\S_\N\V_\N\|_F \notag
	\end{align}
	where \eqref{eq:app_mult} follows from the condition that \eqref{eq:multi} holds with error parameter $\sqrt{\epsilon/c}$. 
	Also, by \eqref{eq:multi} holds with error parameter $\sqrt{\epsilon/r}$, we have
	\begin{align*}
		\|\U_\M^{\perp}(\U_\M^{\perp})^T\A\S_{\N}^T\S_\N\V_\N - \U_\M^{\perp}(\U_\M^{\perp})^T\A\V_\N\|_F \leq \frac{\sqrt{\epsilon}}{\sqrt{r}}\|\U_\M^{\perp}(\U_\M^{\perp})^T\A\|_F\|\V_\N\|_F = \sqrt{\epsilon}\|\U_\M^{\perp}(\U_\M^{\perp})^T\A\|_F
	\end{align*}
	Therefore, we obtain
	\begin{equation}
	\|\U_\M^{\perp}(\U_\M^{\perp})^T\A\S_{\N}^T\S_\N\V_\N\|_F \leq (1+\sqrt{\epsilon})\|\U_\M^{\perp}(\U_\M^{\perp})^T\A\|_F.
	\end{equation}
	Now, we get 
	\begin{align*}
	\|(\S_\M\U_\M)^T\S_{\M}\U_\M^{\perp}(\U_\M^{\perp})^T\A\S_{\N}^T(\V_\N^{T}\S_\N^T)^T\|_F \leq 2\sqrt{\epsilon}\|\U_\M^{\perp}(\U_\M^{\perp})^T\A\|_F.
	\end{align*}
	
	For $\|(\S_\M\U_\M)^T\S_{\M}\U_\M\U_\M^T\A\V_\N^{\perp}(\V_{\N}^{\perp})^T\S_{\N}^T(\V_\N^{T}\S_\N^T)^T\|_F$, we have,
	\begin{align}
	&\|(\S_\M\U_\M)^T\S_{\M}\U_\M\U_\M^T\A\V_\N^{\perp}(\V_{\N}^{\perp})^T\S_{\N}^T(\V_\N^{T}\S_\N^T)^T\|_F\notag\\
	\leq&\|(\S_\M\U_\M)^T\S_{\M}\U_\M\U_\M^T\|_2\|\A\V_\N^{\perp}(\V_{\N}^{\perp})^T\S_{\N}^T(\V_\N^{T}\S_\N^T)^T\|_F\notag\\
	\leq &(1+0.5)\|\A\V_\N^{\perp}(\V_{\N}^{\perp})^T\S_{\N}^T(\V_\N^{T}\S_\N^T)^T\|_F\label{eq:uc_embed}\\
	\leq&2\frac{\sqrt{\epsilon}}{\sqrt{r}}\|\V_\N\|_F\|\A\V_\N^{\perp}(\V_{\N}^{\perp})^T\|_F \label{eq:r_fnorm_kp}\\
	\leq&2\sqrt{\epsilon}\|\A\V_\N^{\perp}(\V_{\N}^{\perp})^T\|_F, \notag
	\end{align}
	where \eqref{eq:uc_embed} follows from the property of subspace embedding property with error parameter $0.5$ and \eqref{eq:r_fnorm_kp} is because the condition that Equation~\eqref{eq:multi} holds with error parameter $\sqrt{\epsilon/r}$.
	
	Thus, we have 
	\begin{align*}
	&\|[(\S_\M\U_\M)^T\S_{\M}\A^\vdash\S_{\N}^T(\V_\N^{T}\S_\N^T)^T]\|_F \\
	\leq& 2\sqrt{\epsilon}\|\U_\M^{\perp}(\U_\M^{\perp})^T\A\|_F +2\sqrt{\epsilon}\|\A\V_\N^{\perp}(\V_{\N}^{\perp})^T\|_F\\
	\leq&4\sqrt{\epsilon}\|\A-\M\X^{*}\N\|_F,
	\end{align*}
	where last inequality follow from the fact that $\|\U_\M^{\perp}(\U_\M^{\perp})^T\A\|_F \leq \|\A-\M\X^{*}\N\|_F$ and $\|\A\V_\N^{\perp}(\V_{\N}^{\perp})^T\|_F\leq \|\A-\M\X^{*}\N\|_F$.
	Thus, we have,
	\begin{align*}
	\Z \leq & \sigma_{\min}^{-2}(\S_\M\U_\M)\sigma_{\min}^{-2}(\S_\N\V_\N)\cdot\|[(\S_\M\U_\M)^T\S_{\M}\A^\vdash\S_{\N}^T(\V_\N^{T}\S_\N^T)^T]\|_F \\
	\leq & 4\sqrt{\epsilon}\|[(\S_\M\U_\M)^T\S_{\M}\A^\vdash \S_{\N}^T(\V_\N^{T}\S_\N^T)^T]\|_F \\
	\leq &16\sqrt{\epsilon}\|\A-\M\X^{*}\N\|_F,
	\end{align*}
	where the second inequality is because subspace embedding properyt of $\S_\M$ and $\S_\N$. It holds that 
	\[
	\sigma^{-2}_{\min} (\S_\M\U_\M) \leq \frac{1}{(1-0.5) \sigma_{\min}^2(\U_\M)} = 2,
	\] 
	and 
	\[
	\sigma^{-2}_{\min} (\S_\N\V_\N) \leq \frac{1}{(1-0.5) \sigma_{\min}^2(\V_\N)} = 2.
	\]
	Finally, we reach that
	\begin{align*}
	\|\A-\M\hat{\X}\N\|_F^{2}  = & \|\A-\M\X^*\N\|_F^{2}+\|\M\X^*\N-\M\hat{\X}\N\|_F^{2} \\
	=& \|\A-\M\X^*\N\|_F^{2} + \|\U_\M\Z\V_\N^{T}\|_F^{2} \\
	\leq& \|\A-\M\X^*\N\|_F^{2} + 256\epsilon\|\A-\M\X^*\N\|_F^{2}\\
	=&(1+256\epsilon)\|\A-\M\X^*\N\|_F^{2}.
	\end{align*}
	By rescaling $\epsilon$, we get the result.
\end{proof}

\section{Conclusion}

In this paper we have studied fast generalized matrix approximation using sketching techniques. We have given a tighter bound of reduced dimensions $s_c$ and $s_r$ to reach a $(1+\epsilon)$ error bound and obtained an $\OM(\nnz(\M)+\nnz(\N))$ generalized matrix approximation. 

\bibliographystyle{plainnat}
\bibliography{referee}

\newpage
\appendix
\section{Important Lemmas} \label{app:lemma}
\begin{lemma}\label{lem:embed_add}
	$\G\in\RB^{\ell\times t}$ and $\Pii\in\RB^{t\times m}$ are $\epsilon$-subspace embedding matrices for $k$-subspace. $\A\in\RB^{m\times k}$ has at most rank $k$. Then, for any vector $\x\in\RB^{k}$, it holds that
	\[
	\|\G\Pii\A\x\|_2^2 = (1\pm 2\epsilon)\|\A\x\|_2^2.
	\] 
\end{lemma}
\begin{proof}
	Since $\G\in\RB^{\ell\times n}$ and $\Pii\in\RB^{n\times m}$ are $\epsilon$-subspace embedding matrix for $k$-subspace, we have 
	\begin{align*}
		\|\G\Pii\A\x\|_2^2 = (1\pm\epsilon)\|\Pii\A\x\|_2^2 = (1\pm\epsilon)^2\|\A\x\|_2^2 = (1\pm2\epsilon)\|\A\x\|_2^2,
	\end{align*}
	where the last equality omit the high order $\epsilon^2$.
\end{proof}

\begin{lemma}[\cite{boutsidis2015optimal,clarkson2013low}] \label{thm:fnorm_kp}
	Given $\A\in\RBmn$, there is an $s = \Ome(\epsilon^{-2})$ so that  for an $s \times m$ sparse embedding matrix $\S$ or an $s \times m$ matrix $\S$ of i.i.d.\ normal random variables with variance $1/s$, Then with high probability,
	\begin{equation}
	\|\S\A\|_F^2 = (1\pm\epsilon)\|\A\|_F^2. \label{eq:S_fnorm_kp}
	\end{equation}	
\end{lemma}

\begin{lemma}\label{lem:mult_add}
	If $\G\in\RB^{\ell\times t}$ and $\Pii\in\RB^{t\times m}$ can be used to approximate matrix products with error parameter $\epsilon$ i.e. Equation~\eqref{eq:multi} holding, besides $\Pii$ can keep the Frobenius norm of matrix with error parameter $\epsilon_0 = 1$, i.e Equation~\eqref{eq:S_fnorm_kp} holding with error parameter $1$, then given  $\A\in\RBmn$ and $\B\in\RB^{m\times k}$, we have 
	\[
	\|\A^{T}\Pii^T\G^T\G\Pii\B - \A^T\B\|_F \leq  5\epsilon\|\A\|_F\|\B\|_F
	\] 
\end{lemma}
\begin{proof}
	Since  $\G\in\RB^{\ell\times n}$ and $\Pii\in\RB^{n\times m}$ can be used to approximate matrix products,  we have 
	\begin{align*}
		&\|\A^{T}\Pii^T\G^T\G\Pii\B - \A^T\B\|_F \\
		= & \|\A^{T}\Pii^T\G^T\G\Pii\B - \A^{T}\Pii^T\Pii\B +\A^{T}\Pii^T\Pii\B- \A^T\B\|_F \\
		\leq&  \|\A^{T}\Pii^T\G^T\G\Pii\B - \A^{T}\Pii^T\Pii\B\|_F + \|\A^{T}\Pii^T\Pii\B- \A^T\B\|_F\\
		\leq& \epsilon(\|\Pii\A\|_F\|\Pii\B\|_F + \|\A\|_F\|\B\|_F)\\
		\leq& 5\epsilon\|\A\|_F\|\B\|_F,
	\end{align*}
\end{proof}
where the last inequality is because $\|\Pii\A\|_F \leq (1+1)\|\A\|_F$ and $\|\Pii\B\|_F \leq (1+1)\|\B\|_F$.
\end{document}